\documentclass[12pt,leqno,oneside]{amsart}
\usepackage{amssymb}
\usepackage{mathrsfs}
\usepackage{graphicx,color}
\usepackage{newtxtext}
\usepackage[varg]{newtxmath}
\usepackage[shortlabels]{enumitem}
\theoremstyle{plain} 
\newtheorem{theorem}{Theorem}[section]
\newtheorem{corollary}[theorem]{Corollary}
\newtheorem{proposition}[theorem]{Proposition}
\newtheorem{lemma}[theorem]{Lemma}

\theoremstyle{definition} 

\newtheorem{remark}[theorem]{Remark}
\newtheorem{example}[theorem]{Example}

\DeclareMathOperator{\Div}{div}

\newcommand{\Z}{\mathbb{Z}}

\newcommand{\R}{\mathbb{R}}
\newcommand{\Sp}{\mathbb{S}}

\usepackage{constants}
\newconstantfamily{eps}{symbol=\varepsilon}
\numberwithin{equation}{section}
\usepackage{bm}
\renewcommand{\vec}[1]{\bm{#1}}
\usepackage{hyperref}

\title{A curve shortening equation with time-dependent mobility
related to grain boundary motions}

\author{Masashi Mizuno}
\address[Masashi Mizuno]%
{Department of Mathematics, College of Science and Technology, Nihon
University, 1-8-14 Kanda-Surugadai, Chiyoda-Ku, Tokyo 101-8308, JAPAN}
\email{mizuno.masashi@nihon-u.ac.jp}

\author{Keisuke Takasao}
\address[Keisuke Takasao]%
{Department of Mathematics/Hakubi Center, Kyoto University, Kitashirakawa-Oiwakecho Sakyo, Kyoto 606-8502, JAPAN}
\email{k.takasao@math.kyoto-u.ac.jp}



\allowdisplaybreaks[1]
\begin{document}


\begin{abstract}
 A curve shortening equation related to the
 evolution of grain boundaries is presented.
 This equation is derived from the grain boundary energy by applying
 the maximum dissipation principle.
 Gradient estimates and large time asymptotic behavior of solutions are
 considered.
 In the proof of these results, one key ingredient is a new
 weighted monotonicity formula that incorporates a time-dependent
 mobility.

%
%

\end{abstract}

\maketitle

\section{Introduction}
We study a curve shortening equation related to the evolution of grain
boundaries. Most materials have a polycrystalline microstructure
composed of a myriad of tiny single crystalline grains separated by
grain boundaries. Many experimental results indicate that the microscale
structure of the grain boundaries is strongly related to the macroscale
properties of the material composed of these grain boundaries.

Mathematical modeling of the grain boundaries was first studied by
Mullins and Herring~\cite{doi:10.1007-978-3-642-59938-5_2,
doi:10.1063-1.1722511,doi:10.1063-1.1722742}. In particular when the
grain boundary energy depends only on the length and shape of these
grain boundaries, a curve shortening equation or a mean curvature flow
equation is obtained. Both equations are quasilinear and underlie
important problems in geometric analysis; hence there is a diversity of
research looking into these problems.

However, from the perspective of research on grain boundaries, it is
also important to treat other state variables. For instance, grain
boundaries are regarded as some singularity in lattice orientation of
each grain. Kinderlehrer-Liu~\cite{MR1833000} introduced
misorientations, which are the differences in lattice orientation of two
grains separated by a grain boundary, as a parameter in the expression
for the grain boundary energy. They derived geometric evolution
equations based on the maximal dissipation principle.
Epshteyn-Liu-Mizuno~\cite{arXiv:1903.11512, arXiv:1910.08022} considered
the case that the misorientation depends on the time and demonstrated
the local existence of network solutions provided the grain boundaries
are straight line segments. Nevertheless, the interaction between
curvature and misorientation is not well-known.

\begin{figure}
 \centering
 \includegraphics[width=6cm]{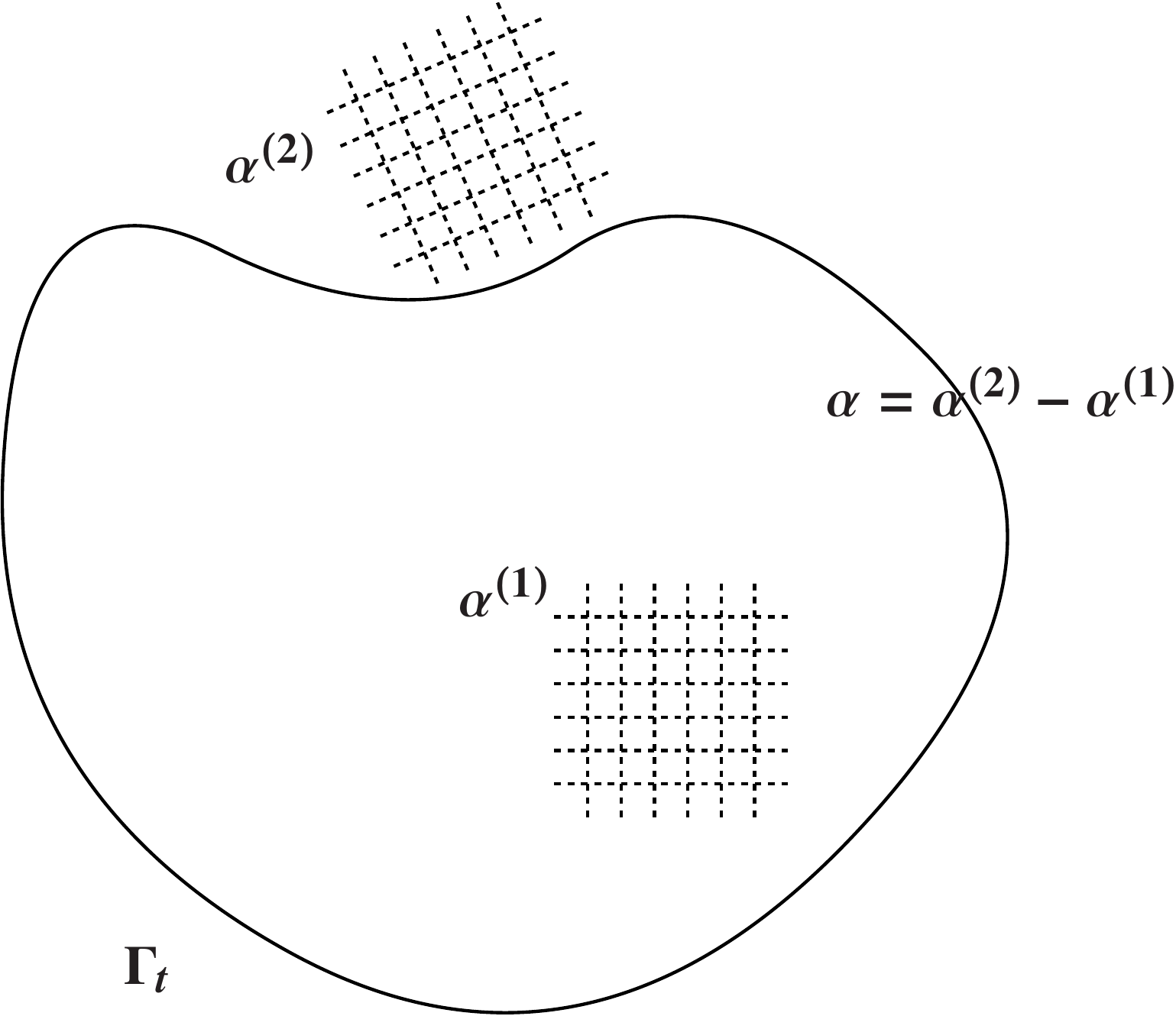}
 \caption{Let $\Gamma_t\subset\R^2$ be a smooth Jordan
 curve, and let $\alpha^{(1)}$, $\alpha^{(2)}$ be time-dependent unknown
 functions, called orientations. The difference
 $\alpha=\alpha^{(2)}-\alpha^{(1)}$ is called the misorientation.}
 \label{fig:1.1}
\end{figure}

In this article, we study the grain boundary energy that include
time-dependent misorientations as a state variable.  First, we consider
a smooth Jordan curve $\Gamma_t\subset\R^2$ as a grain boundary, with
$v_n$ and $\kappa$ denoting the normal velocity and the curvature of
$\Gamma_t$, respectively. We assume that the misorientation
$\alpha(t)=\alpha^{(2)}(t)-\alpha^{(1)}(t)$ depends on the time and is
independent of the position vector of the grain boundary(See Figure
\ref{fig:1.1}). Taking the grain boundary energy as
\begin{equation*}
 \int_{\Gamma_t}
  \sigma(\alpha)\,d\mathscr{H}^1
  =
  \sigma(\alpha)
   |\Gamma_t|,
\end{equation*}
we derived a system of evolution equations obtained from the maximum
dissipation principle:
\begin{equation}
 \label{eq:1.1}
  \left\{
  \begin{aligned}
   v_n
   &=
   \mu
   \sigma(\alpha)
   \kappa,&
   \quad
   &\text{on}\ \Gamma_t,\ t>0, \\
   \alpha_t
   &=
   -
   \gamma
   \sigma_\alpha(\alpha)
   |\Gamma_t|,&
   \quad
   &t>0.
  \end{aligned}
 \right.
\end{equation}
Here, $\mu $ and $\gamma$ denote positive constants, $\sigma :
\mathbb{R} \to [0,\infty)$ denotes a given smooth function, and $|\Gamma
_t|$ the length of $\Gamma_t$. We present a derivation of \eqref{eq:1.1}
in Section 2. Our system consists of two equations, one being a curve
shortening equation with time-dependent mobility, and the other
describing the evolution of the misorientation. The most significant
difference between the PDE in \cite{MR1833000} and \eqref{eq:1.1} is the
time-dependent misorientation. The evolution of a misorientation was
considered in \cite{arXiv:1903.11512, arXiv:1910.08022}. However, only
the relaxation limit $\mu\rightarrow\infty$ was studied, namely, the
authors employed straight line segments to be grain boundaries. On the
other hand, they considered curved grain boundaries in the derivation of
the system. For this reason, understanding the relationship between the
effect of curvature and the evolution of misorientations is important.

In regard to curve shortening flow, specifically time-independent
misorientations, a solution of \eqref{eq:1.1} exists in a finite time if
the initial data is a Jordan curve.  For example, if $\inf_{\alpha \in
\mathbb{R}} \sigma(\alpha)>C$ and $\tilde \Gamma _0=\{ |x|=R \}$ for
some constants $C>0$ and $R>0$, then the solution $(\tilde \Gamma _t,
\tilde \alpha )$ of \eqref{eq:1.1} with the initial data $(\tilde
\Gamma_0 , \tilde \alpha _0)$ is also a circle and the radius $r (t)$
coincides with
\begin{equation*}
 \sqrt{R^2-2\mu \int _0 ^t \sigma(\tilde \alpha (s))
  ds}.
\end{equation*}
Note that the comparison principle implies $ \Gamma _t \subset \{ |x|
\leq \sqrt{R^2-2\mu C t } \} $ for any solution $(\Gamma _t,\alpha)$
such that $\Gamma _0 \subset \{ |x| \leq R \}$, since
$\{|x|=\sqrt{R^2-2\mu C t }\}$ is a solution of $v_n=\mu C \kappa$.
Therefore, any solution starting from a Jordan curve disappears in a
finite time. In contrast, as for curve shortening flow, the solution is
expected to converge to a straight line under suitable conditions,
although the effects from boundary conditions and junctions also need to
be considered (see Example \ref{example:2.3}). The mean curvature flow
of the graph has been studied in \cite{MR0998603,MR1025164,MR1117150},
but is not well-known in regard to effects concerning the evolving
misorientations. Consequently, to understand the nature of the time
global classical solution of \eqref{eq:1.1}, we consider two unbounded
grains, and their grain boundary represented by a periodic graph (see
\eqref{eq:2.13} below).  In this situation, we study the properties of
the time global solutions.

To obtain the solvability of the system in the graphical setting, a
priori gradient estimates for solutions of our system play an important
role. For the curve shortening equation with constant mobility,
Huisken~\cite{MR1030675} derived the so-called monotonicity formula
(cf. \cite{MR0784476}) and Ecker-Huisken\cite{MR1025164} provided
gradient estimates for the entire graph using Huisken's monotonicity
formula (See also \cite{MR3661017,MR3058702}. Sharp gradient estimates
are given in \cite{MR2099114}). Key ingredients of Huisken's
monotonicity formula are the properties of the standard backward heat
kernel. We derive the weighted monotonicity formula in similar manner as
for Huisken's formula(cf. Ecker~\cite[Theorem 4.13]{MR2024995}) for the
curve shortening equation with a \emph{time-dependent} mobility (see
Theorem \ref{thm:3.1} below). Then, using the weighted monotonicity
formula we obtain gradient estimates and the global existence of
solutions for the problem (see Theorem \ref{thm:4.2} and Theorem
\ref{thm:4.5} below).  Our new argument is to replace the standard
backward heat kernel with one with time-dependent thermal
conductivity. Finally, we prove that the time global solution converges
to a straight line exponentially in $C^2$ (see Theorem \ref{thm:5.1}).

The paper is organized as follows. In Section 2, we set up the model and
derive evolution equations using the maximum dissipation principle. We
consider a graph of an unknown function as a grain boundary and derive a
governing equation from the model.  In Section 3, we briefly review
backward heat kernels with time-dependent thermal conductivity. Next, we
obtain the weighted monotonicity identity for our problem. Using this
identity, we derive gradient estimates and the global existence of
solutions to our problem in Section 4. In Section 5, we deduce the large
time asymptotic behavior of the global solution.

\section{Derivation of the system}
We begin by deriving the governing equations of our systems from the
energy dissipation principle. This approach is taken from
\cite{arXiv:1903.11512,arXiv:1910.08022}, without the effect of the
triple junction drag. We consider a single grain boundary $\Gamma_t$
represented by point vector $\vec{\xi}(s,t)\in\R^2$ for $0\leq s\leq1$
and $t>0$. Note that $s$ is \emph{not} necessarily the arclength
parameter. To understand the relationship between misorientations and
the effect of curvature, we impose the periodic boundary condition,
specifically $\vec{\xi}(0,t)=\vec{\xi}(1,t)$ and
$\vec{\xi}_s(0,t)=\vec{\xi}_s(1,t)$ for $t>0$. We denote a tangent
vector by $\vec{b}=\vec{\xi}_s$ and a normal vector by
$\vec{n}=R\vec{b}$ where $R$ is a matrix describing an anti-clockwise
rotation through angle $\pi/2$. Again we remark that the tangent vector
$\vec{b}$ and the normal vector $\vec{n}$ are not necessarily unit
vectors because in general $s$ is not the arclength parameter.

\begin{figure}
 \centering \includegraphics[width=6cm]{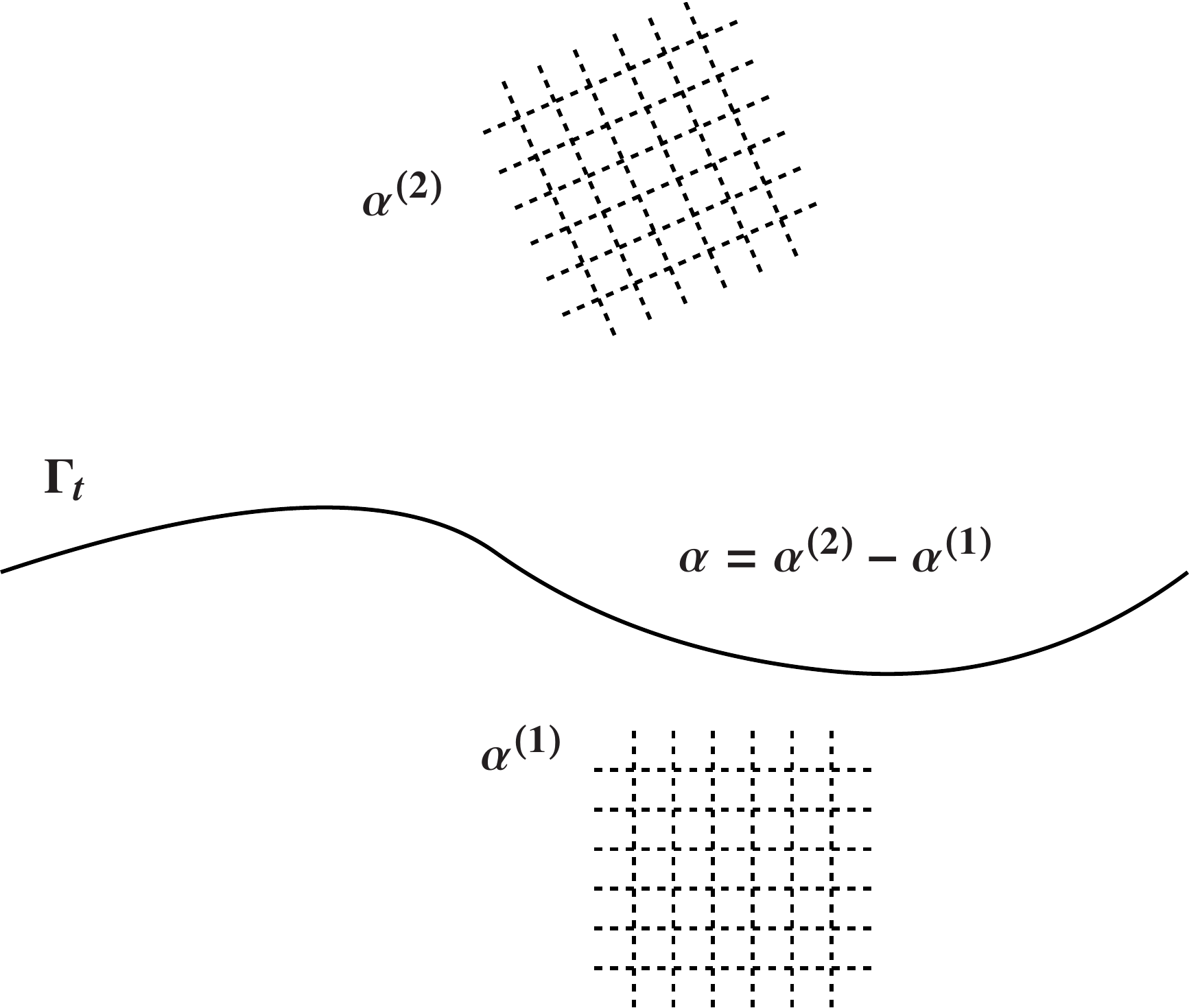} \caption{Model
 of a single grain boundary $\Gamma_t$. State variables $\alpha^{(1)}$
 and $\alpha^{(2)}$ represent the lattice orientations of the grains.
 State variable $\alpha=\alpha^{(2)}-\alpha^{(1)}$ defines the
 misorientation on the grain boundary $\Gamma_t$.}  \label{fig:2.1}
\end{figure}

Next, we let $\alpha=\alpha(t)$ be the lattice misorientation on the
grain boundary $\Gamma_t$. We assume that the lattice misorientation
$\alpha$ depends on time $t$, but is independent of parameter $s$. We
consider the normal vector $\vec{n}$ and the lattice misorientation
$\alpha$ as state variables so we define the interfacial grain boundary
energy density of $\Gamma_t$ as
\begin{equation*}
 \sigma
  =
  \sigma(\vec{n},\alpha)
  \geq0.
\end{equation*}
Thus the total grain boundary energy of the system $\Gamma_t$ is written
\begin{equation}
 \label{eq:2.1}
  E(t)
  =
  \int_{\Gamma_t}
  \sigma(\vec{n},\alpha)\,d\mathscr{H}^1
  =
  \int_{0}^{1}
  \sigma(\vec{n}(s,t),\alpha(t))|\vec{b}(s,t)|\,ds,
\end{equation}
where $\mathscr{H}^1$ is the $1$-dimensional Hausdorff measure and
$|\cdot|$ is the standard Euclidean vector norm on $\R^2$. Next, we
assume that $\sigma$ is a non-negative smooth function and positively
homogeneous of degree $0$ in $\vec{n}$.

Let us now derive the grain boundary motion from the dissipation
principle of the total grain boundary energy \eqref{eq:2.1}. Let $\hat{\
}$ be the normalization operator of vectors, e.g.,
$\hat{\vec{b}}=\frac{\vec{b}}{|\vec{b}|}$. Next, we compute the
dissipation rate of the total grain boundary energy $E(t)$ at time $t$,
\begin{equation}
 \label{eq:2.2}
  \begin{split}
   \frac{d}{dt}E(t)
   &=
   \int_{0}^{1}
   \sigma_{\vec{n}}
   \cdot\frac{d\vec{n}}{dt}|\vec{b}|\,ds
   +
   \int_{0}^{1}\sigma
   \frac{\vec{b}}{|\vec{b}|}
   \cdot\frac{d\vec{b}}{dt}\,ds
   +
   \int_{0}^{1}\sigma_\alpha
   \frac{d\alpha}{dt}|\vec{b}|\,ds
   \\
   &=
   \int_{0}^{1}
   \left(
   |\vec{b}|
   \mathstrut^t\!R\sigma_{\vec{n}}
   +
   \sigma
   \hat{\vec{b}}
   \right)
   \cdot\frac{d\vec{b}}{dt}\,ds
   +
   \int_{0}^{1}
   \sigma_\alpha
   \frac{d\alpha}{dt}|\vec{b}|\,ds.
  \end{split}
\end{equation}
Now, consider a polar angle $\theta$ for $\vec{n}$ and set
$\vec{n}=|\vec{n}|(\cos\theta,\sin\theta)$. Since $\sigma$ is positively
homogeneous of degree $0$ in $\vec{n}$, we have
\begin{equation}
 \label{eq:2.7}
  \begin{aligned}
   0
   &=
   \frac{d}{d\lambda}
   \sigma(\lambda\vec{n},\alpha)
   \bigg|_{\lambda=1}
   =
   \sigma_{\vec{n}}
   (\vec{n},\alpha)\cdot\vec{n},&\quad
   \mathstrut^t\!R
   \sigma_{\vec{n}}
   &=
   (\mathstrut{}^t\!R
   \sigma_{\vec{n}}\cdot\hat{\vec{n}})\hat{\vec{n}}, \\
   \sigma_\theta
   &:=
   \frac{d}{d\theta}\sigma(\vec{n},\alpha)
   =
   |\vec{n}|\mathstrut{}^t\!R\sigma_{\vec{n}}\cdot\hat{\vec{n}},
   &\quad
   \sigma_\theta
   \hat{\vec{n}}
   &=
   |\vec{b}|
   \mathstrut^t\!R
   \sigma_{\vec{n}},
  \end{aligned}
\end{equation}
and thus, we define the vector $\vec{T}$ known as the line tension
vector,
\begin{equation*}
 \vec{T}
  :=\sigma_\theta
  \hat{\vec{n}}
  +
  \sigma
  \hat{\vec{b}}
  =
  |\vec{b}|
   \mathstrut^t\!R\sigma_{\vec{n}}
   +
   \sigma
   \hat{\vec{b}}.
\end{equation*}
Next, using a change of variable
\begin{equation}
 \frac{d\vec{b}}{dt}
 =
 \frac{d}{ds}
 \frac{d\vec{\xi}}{dt},
\end{equation}
we rewrite \eqref{eq:2.2} as
\begin{equation}
 \label{eq:2.3}
  \frac{d}{dt}E(t)
  =
  \int_{0}^{1}
  \vec{T}
  \cdot
  \frac{d}{ds}
  \frac{d\vec{\xi}}{dt}\,ds
  +
  \int_{0}^{1}
  \sigma_\alpha
  \frac{d \alpha}{dt}|\vec{b}|\,ds
  =
  -
  \int_{0}^{1}
  \vec{T}_s
  \cdot
  \frac{d\vec{\xi}}{dt}\,ds
  +
  \int_{0}^{1}
  \sigma_\alpha
  \frac{d\alpha}{dt}|\vec{b}|\,ds
\end{equation}
from the periodic condition $\vec{b}(0,t)=\vec{b}(1,t)$.

For the reader's convenience, we recall a property of the derivative of
the line tension vector $\vec{T}$.

\begin{lemma}[cf. \cite{MR1833000}]
 \label{lem:2.1}
 Let $\kappa$ be the curvature of $\Gamma_t$. Then
 \begin{equation}
  \label{eq:2.4}
   \vec{T}_s
   =
   |\vec{b}|
   (\sigma_{\theta\theta}+\sigma)\kappa\hat{\vec{n}}.
 \end{equation}
\end{lemma}

\begin{proof}
 Denote $\partial_{\Gamma_t}=\frac{1}{|\vec{b}|}\partial_s$, which is the
 arc-length derivative along with $\Gamma_t$. From the
 Frenet-Serret formula, we obtain
 \begin{equation}
  \label{eq:2.5}
   \hat{\vec{b}}_s
   =
   |\vec{b}|
   \partial_{\Gamma_t}
   \hat{\vec{b}}
   =
   |\vec{b}|
   \kappa
   \hat{\vec{n}},
   \quad
   \hat{\vec{n}}_s
   =
   |\vec{b}|
   \partial_{\Gamma_t}
   \hat{\vec{n}}
   =
   -|\vec{b}|
   \kappa
   \hat{\vec{b}}.
 \end{equation}
 Hence, we obtain,
 \begin{equation}
  \label{eq:2.6}
   \begin{split}
    \vec{T}_s
    &=
    \left(
    \sigma_{\vec{n}\theta}\cdot
    \vec{n}_s
    \right)
    \hat{\vec{n}}
    +
    \sigma_\theta
    \hat{\vec{n}}_s
    +
    \left(
    \sigma_{\vec{n}}
    \cdot
    \vec{n}_s
    \right)
    \hat{\vec{b}}
    +
    \sigma
    \hat{\vec{b}}_s \\
    &=
    \left(
    \mathstrut^t\!R\sigma_{\vec{n}\theta}
    \cdot
    \vec{b}_s
    +
    |\vec{b}|\sigma\kappa
    \right)
    \hat{\vec{n}}
    +
    \left(
    -|\vec{b}|\sigma_\theta\kappa+
    \mathstrut^t\!R\sigma_{\vec{n}}
    \cdot
    \vec{b}_s
    \right)
    \hat{\vec{b}}.
   \end{split}
 \end{equation}
 Since $\sigma$ and $\sigma_\theta$ are positively
 homogeneous of degree $0$ in $\vec{n}$, as the similar calculation on
 \eqref{eq:2.7}, we have
 \begin{equation}
  \label{eq:2.8}
   \sigma_\theta
   \hat{\vec{n}}
   =
   |\vec{b}|
   \mathstrut^t\!R
   \sigma_{\vec{n}},\quad
   \sigma_{\theta\theta}
   \hat{\vec{n}}
   =
   |\vec{b}|
   \mathstrut^t\!R
   \sigma_{\vec{n}\theta}.
 \end{equation}
 Using the orthogonal relation $\vec{b}\cdot\hat{\vec{n}}=0$ and the
 Frenet-Serret formula \eqref{eq:2.5}, we obtain
 $\vec{b}_s\cdot\hat{\vec{n}} = -\vec{b}\cdot\hat{\vec{n}}_s =
 |\vec{b}|^2\kappa$. Thus, from \eqref{eq:2.8}
 \begin{equation*}
  \begin{split}
   \mathstrut^t\!R\sigma_{\vec{n}\theta}
   \cdot
   \vec{b}_s
   +
   |\vec{b}|\sigma\kappa
   &=
   \frac{1}{|\vec{b}|}
   \sigma_{\theta\theta}\hat{\vec{n}}
   \cdot
   \vec{b}_s
   +
   |\vec{b}|\sigma\kappa
   =
   |\vec{b}|
   \left(
   \sigma_{\theta\theta}
   +
   \sigma
   \right)
   \kappa,
   \\
   -
   |\vec{b}|\sigma_\theta\kappa
   +
   \mathstrut^t\!R\sigma_{\vec{n}}
   \cdot
   \vec{b}_s
   &=
   -
   |\vec{b}|\sigma_\theta\kappa
   +
   \frac{1}{|\vec{b}|}
   \sigma_\theta
   \hat{\vec{n}}
   \cdot
   \vec{b}_s
   =
   0
  \end{split}
 \end{equation*}
 and hence we derive \eqref{eq:2.4}.
\end{proof}
To ensure that the whole system is dissipative, i.e.
\begin{equation*}
 \frac{d}{dt}E(t)\leq 0,
\end{equation*}
we impose the so called Mullins equation or the curve shortening
equation for the evolution of the grain boundary $\Gamma_t$. From Lemma
\ref{lem:2.1}, $\vec{T}_s$ is proportional to the normal vector on
$\Gamma_t$ and therefore we impose
\begin{equation}
 \label{eq:2.9}
 v_n
  =
  \mu\partial_{\Gamma_t}\vec{T}\cdot\hat{\vec{n}}
  =
  \mu(\sigma_{\theta\theta}+\sigma)\kappa\quad
  \text{on}\ \Gamma_t,
\end{equation}
where $v_n$ denotes the normal velocity vector of $\Gamma_t$ and $\mu>0$
a positive mobility constant. Note that equation \eqref{eq:2.9} may
be derived from the variation of the energy $E$ with respect to the
curve $\vec{\xi}$. Indeed, for any test function $\vec{\phi}\in
C^\infty_0(0,1)$,
\begin{equation}
 \begin{split}
  \frac{\delta E}{\delta \vec{\xi}}[\vec{\phi}]
  &=
  \int_0^1
  \left(
  (\sigma_{\vec{n}}(\vec{n},\alpha)
  \cdot R\vec{\phi}_s)|\vec{b}|
  +\sigma(\vec{n},\alpha)
  \hat{\vec{b}}\cdot\vec{\phi}_s
  \right)
  \,ds \\
  &=
  \int_0^1
  \left(
  |\vec{b}|\mathstrut^{t}\!R
  \sigma_{\vec{n}}(\vec{n},\alpha)
  +\sigma(\vec{n},\alpha)
  \hat{\vec{b}}
  \right)
  \cdot
  \vec{\phi}_s
  \,ds \\
  &=
  -\int_{\Gamma_t}
  \partial_{\Gamma_t}
  \left(
  |\vec{b}|\mathstrut^{t}\!R
  \sigma_{\vec{n}}(\vec{n},\alpha)
  +\sigma(\vec{n},\alpha)
  \hat{\vec{b}}
  \right)
  \cdot
  \vec{\phi}
  \,d\mathscr{H}^1,
 \end{split}
\end{equation}
thus \eqref{eq:2.9} is turned into
\begin{equation*}
 \frac{d\vec{\xi}}{dt}
  =
  -\mu
  \frac{\delta E}{\delta \vec{\xi}}.
\end{equation*}
Since $v_n=\vec{\xi}_t\cdot\hat{\vec{n}}$, we obtain
\begin{equation}
 \vec{T}_s\cdot\frac{d\vec{\xi}}{dt}
  =
  \frac{1}{\mu}
  |v_n|^2
  |\vec{b}|
  \geq
  0.
\end{equation}
Next, we consider the law underlying evolution of lattice misorientations.
Since $\alpha$ is independent of the parameter $s$,
\begin{equation*}
 \int_{0}^{1}
  \sigma_\alpha
  \frac{d\alpha}{dt}|\vec{b}|\,ds
  =
  \frac{d\alpha}{dt}
  \int_{0}^{1}
  \sigma_\alpha
  |\vec{b}|\,ds
  =
  \frac{d\alpha}{dt}
  \int_{\Gamma_t}
  \sigma_\alpha
  \,d\mathscr{H}^1,
\end{equation*}
hence for a constant $\gamma>0$, we impose the following relation for
the rate of change of the lattice misorientation;
\begin{equation}
 \label{eq:2.10}
 \frac{d\alpha}{dt}
  =
  -\gamma
  \int_{\Gamma_t}
  \sigma_\alpha
  \,d\mathscr{H}^1,
\end{equation}
to ensure the whole system is dissipative, namely
$\frac{d}{dt}E(t)\leq0$. Note that our proposed equation \eqref{eq:2.10}
can be derived from the variation of the energy $E$ with respect to lattice
misorientation $\alpha$. Indeed for any number $\xi\in\R$,
\begin{equation*}
 \frac{\delta E}{\delta \alpha}[\xi]
  =
  \frac{d}{d\varepsilon}\bigg|_{\varepsilon=0}
  \int_{0}^{1}
  \sigma(\vec{n},\alpha+\varepsilon\xi)
  |\vec{b}|\,ds
 =
  \xi
  \int_{0}^{1}
  \sigma_\alpha(\vec{n},\alpha)|\vec{b}|\,ds,
\end{equation*}
thus \eqref{eq:2.10} becomes
\begin{equation}
 \frac{d\alpha}{dt}
  =
  -\gamma
  \frac{\delta E}{\delta \alpha}.
\end{equation}
Now, substituting equations \eqref{eq:2.9} and \eqref{eq:2.10} in the
rate of change for the total energy \eqref{eq:2.3}, we find that the
whole system is dissipative, namely
\begin{equation}
 \label{eq:2.21}
 \frac{d}{dt}E(t)
  =
  -\frac{1}\mu
  \int_{\Gamma_t}
  |v_n|^2\,d\mathscr{H}^1
  -\frac{1}{\gamma}
  \left|
   \frac{d\alpha}{dt}
  \right|^2
  \leq0.
\end{equation}

\begin{remark}
 We emphasize in \eqref{eq:2.21} that the evolving misorientation
 $\alpha$ has a dissipative structure. See also
 \cite{arXiv:1903.11512}. In contrast, the misorientation is a fixed
 parameter in \cite{MR1833000}.
\end{remark}

We next consider the grain boundary motion for the isotropic case. The
grain boundary energy density $\sigma$ is independent of the normal
vector $\vec{n}$. Then, the equations \eqref{eq:2.9} and \eqref{eq:2.10}
become
\begin{equation}
 \label{eq:2.12}
 \left\{
  \begin{aligned}
   v_n
   &=
   \mu
   \sigma(\alpha)
   \kappa,&
   \quad
   &\text{on}\ \Gamma_t,\ t>0, \\
   \alpha_t
   &=
   -
   \gamma
   \sigma_\alpha(\alpha)
   |\Gamma_t|,&
   \quad
   &t>0.
  \end{aligned}
 \right.
\end{equation}
Imposing the periodic boundary
condition, we put $\mathbb{T}:=\R/\Z$ and write $\Gamma_t$ as a graph of
an unknown function $u=u(x,t)$ on $\mathbb{T}\times[0,\infty)$, namely
\begin{equation}
 \vec{\xi}(x,t)=(x,u(x,t)),\quad  x\in\mathbb{T},\quad t>0.
\end{equation}
With the initial data $\vec{\xi}(x,0)=(x,u_0(x))$,
$\alpha(0)=\alpha_0\in\R$, and the periodic boundary
condition $\vec{\xi}(0,t)=\vec{\xi}(1,t)$,
$\vec{\xi}_s(0,t)=\vec{\xi}_s(1,t)$, equation \eqref{eq:2.12} becomes
\begin{equation}
 \label{eq:2.13}
 \left\{
  \begin{aligned}
   \frac{u_t}{\sqrt{1+|u_x|^2}}
   &=
   \mu\sigma(\alpha)
   \left(
   \frac{u_x}{\sqrt{1+|u_x|^2}}
   \right)_x,&
   \quad
   &x\in\mathbb{T},\ t>0, \\
   \alpha_t
   &=
   -
   \gamma
   \sigma_\alpha(\alpha)
   |\Gamma_t|, &
   \quad
   &t>0,
   \\
   u(0,t)
   &=
   u(1,t),\quad
   u_x(0,t)
   =
   u_x(1,t),
   &\quad
   &t>0, \\
   u(x,0)&=
   u_0(x),&\quad
   &x\in \mathbb{T},\\
   \alpha(0)&=
   \alpha_0.
  \end{aligned}
 \right.
\end{equation}
Indeed, the normal velocity $v_n$ and the curvature
$\kappa$ are given by
\begin{equation*}
 \begin{split}
   v_n
   &=
   \vec{\xi}_t\cdot\hat{\vec{n}}
   =
   (0,u_t)\cdot
   \left(
   \frac{1}{\sqrt{1+|u_x|^2}}(-u_x,1)
   \right)
   =
   \frac{u_t}{\sqrt{1+|u_x|^2}}, \\
   \kappa
   &=
   \partial_{\Gamma_t}
   \hat{\vec{b}}
   \cdot
   \hat{\vec{n}}
   =
   \frac{1}{\sqrt{1+|u_x|^2}}
   \left(
   \frac{1}{\sqrt{1+|u_x|^2}}
   (1,u_x)
   \right)_x
   \cdot
   \left(
   \frac{1}{\sqrt{1+|u_x|^2}}
   (-u_x,1)
   \right) \\
   &=
   \left(
   \frac{u_x}{\sqrt{1+|u_x|^2}}
   \right)_x.
 \end{split}
\end{equation*}
From \eqref{eq:2.1}, the associated total grain boundary energy $E(t)$ is
given by
\begin{equation}
 E(t)
  =
  \int_{\Gamma_t}\sigma(\alpha)
  =
  \sigma(\alpha)
  \int_0^1 \sqrt{1+|u_x|^2}\,dx.
\end{equation}

\begin{proposition}[Free energy dissipation]
 Let $u$ be a solution of \eqref{eq:2.13}. Then
 \begin{equation}
  \label{eq:2.14}
  \frac{dE}{dt}
  =
   -\frac1\gamma
   |\alpha_t|^2
   -\frac1\mu
   \int_0^1
   \left(
    \frac{u_t}{\sqrt{1+|u_x|^2}}
    \right)^2
   \sqrt{1+|u_x|^2}\,dx.
 \end{equation}
\end{proposition}

\begin{proof}
 By direct calculation, we obtain
 \begin{equation}
  \begin{split}
   \frac{dE}{dt}
   &=
   \sigma_\alpha\alpha_t|\Gamma_t|
   +
   \sigma
   \int_0^1
   \frac{u_xu_{xt}}{\sqrt{1+|u_x|^2}}\,dx \\
   &=
   \sigma_\alpha\alpha_t|\Gamma_t|
   -
   \sigma
   \int_0^1
   \left(
   \frac{u_x}{\sqrt{1+|u_x|^2}}
   \right)_x
   u_t
   \,dx \\
   &=
   -\frac1\gamma
   |\alpha_t|^2
   -\frac1\mu
   \int_0^1
   \left(
   \frac{u_t}{\sqrt{1+|u_x|^2}}
   \right)^2
   \sqrt{1+|u_x|^2}\,dx.
  \end{split}
 \end{equation}
\end{proof}

Hereafter, we make two assumptions, first being that the energy density
is strictly positive, namely there exists a positive constant
$\Cl{const:1.1}>0$ such that
\begin{equation}
 \label{eq:1.A1}
  \tag{A1}
  \sigma(\alpha)\geq \Cr{const:1.1}
\end{equation}
for all $\alpha\in\R$. The second is that for $\alpha\in\R$
\begin{equation}
 \label{eq:1.A2}
  \tag{A2}
  \alpha\sigma_\alpha(\alpha)\geq 0.
\end{equation}

\begin{example}
 \label{example:2.3}
 When we consider $\sigma(\alpha)=1+\frac12\alpha^2$, then
 $\Cr{const:1.1}=1$ and we obtain equations:
 \begin{equation*}
  \left\{
   \begin{aligned}
    \frac{u_t}{\sqrt{1+|u_x|^2}}
    &=
    \mu
    \left(
    1+\frac12\alpha^2(t)
    \right)
    \left(
    \frac{u_x}{\sqrt{1+|u_x|^2}}
    \right)_x,&
    \quad
    &x\in(0,1),\ t>0,\\
   \alpha_t
    &=
    -
   \gamma
   \alpha(t)
   |\Gamma_t|, &
   \quad
   &t>0.
   \end{aligned}
 \right.
 \end{equation*}
 For example, $(u,\alpha)=(c_1, c_2 e^{-\gamma t})$ is an explicit solution for any constants $c_1$ and $c_2$.
\end{example}

\section{Weighted monotonicity formula}

Next, we derive a weighted monotonicity formula for \eqref{eq:2.13},
which is useful for gradient estimates. In order to obtain the formula,
we describe the backward heat kernel with time dependent thermal
conductivities and its properties.

\subsection{Backward heat kernels with time-dependent thermal conductivities}
From \eqref{eq:2.12}, we have to consider the fundamental solution of
the heat equation with a time-dependent thermal conductivity. Let us
study
\begin{equation}
 \label{eq:3.1}
 \frac{\partial u}{\partial t}(x,t)
  =
  k'(t)\Delta u(x,t)\quad
  x\in\R^d,\ t>0,
\end{equation}
where $k(t)$ denotes the given thermal conductivity depending on
$t>0$. Taking a change of variable $s=k(t)$, we obtain
\begin{equation*}
 \frac{\partial u}{\partial s}(x,s)
  =
  \Delta u(x,s)\quad
  x\in\R^d,\ s>0.
\end{equation*}
Thus, the fundamental solution of \eqref{eq:3.1} is given by
\begin{equation}
 \frac{1}{(4\pi s)^{d/2}}
  \exp\left(-\frac{|x|^2}{4s}\right)
  =
  \frac{1}{(4\pi k(t))^{d/2}}
  \exp\left(-\frac{|x|^2}{4k(t)}\right).
\end{equation}

Let $k'(t)=\mu\sigma(\alpha(t))$; note that $k'>\mu \Cr{const:1.1}$
by \eqref{eq:1.A1}. For $X_0\in \R^2$ and $t_0>0$, we define the
backward heat kernel $\rho=\rho_{(X_0,t_0)}$ as
\begin{equation}
 \label{eq:3.3}
 \rho(X,t)
  =
  \frac{1}{(4\pi (k(t_0)-k(t)))^{\frac{1}{2}}}
  \exp
  \left(
   -\frac{|X-X_0|^2}{4(k(t_0)-k(t))}
  \right)
  ,\quad
  0<t<t_0,\quad X\in\R^2.
\end{equation}
Then, by direct calculation we get
\begin{equation}
 \begin{split}
  \rho_t
  &=
  \frac{k'(t)}{2(k(t_0)-k(t))}\rho
  -
  \frac{k'(t)|X-X_0|^2}{4(k(t_0)-k(t))^2}\rho, \\
  D\rho
  &=
  -
  \frac{\rho}{2(k(t_0)-k(t))}(X-X_0), \\
  D^2\rho
  &=
  -
  \frac{\rho}{2(k(t_0)-k(t))}I
  +
  \frac{\rho}{4(k(t_0)-k(t))^2}(X-X_0)\otimes(X-X_0),
 \end{split}
\end{equation}
where $X\otimes Y=(x_iy_j)_{1\leq i,j\leq2}$ for $X=(x_1,x_2)$,
$Y=(y_1,y_2)\in\R^2$. Therefore we obtain
\begin{equation}
 \label{eq:3.2}
 \begin{split}
  &\quad
  \rho_t
  +
  \mu\sigma(\alpha(t))
  \frac{(D\rho\cdot\vec{a})^2}{\rho}
  +
  \mu\sigma(\alpha(t))
  ((I-\vec{a}\otimes\vec{a}):D^2\rho) \\
  &=
  \left(
  \frac{k'(t)}{2(k(t_0)-k(t))}\rho
  -
  \frac{k'(t)|X-X_0|^2}{4(k(t_0)-k(t))^2}\rho
  \right)
  +
  \frac{k'(t)\rho}{4(k(t_0)-k(t))^2}((X-X_0)\cdot\vec{a})^2 \\
  &\quad\!
  +
  k'(t)
  \left(
  -
  \frac{\rho}{2(k(t_0)-k(t))}
  +
  \frac{\rho}{4(k(t_0)-k(t))^2}|X-X_0|^2
  -
  \frac{\rho}{4(k(t_0)-k(t))^2}((X-X_0)\cdot\vec{a})^2
  \right) \\
  &=0,
 \end{split}
\end{equation}
for $\vec{a}\in\Sp^1$. We now use the backward heat kernel with
$k'(t)=\mu\sigma(\alpha(t))$ and $k(0)=0$, namely
\begin{equation}
 \label{eq:3.4}
 \rho(X,t)
  :=
  \frac{1}{(4\pi(\Sigma(t_0)-\Sigma(t)))^{\frac{1}{2}}}
  \exp
  \left(
   -\frac{|X-X_0|^2}{4(\Sigma(t_0)-\Sigma(t))}
  \right)
  ,\quad
  0<t<t_0,\quad X\in\R^2,
\end{equation}
where
\begin{equation}
 \Sigma(t)
  :=
  \mu\int_0^t\sigma(\alpha(\tau))\,d\tau.
\end{equation}

\subsection{Weighted monotonicity identity}

The monotonicity formula for the mean curvature flow was derived by
Huisken~\cite{MR1030675} to study asymptotics of blow-up profiles. Ecker
and Huisken~\cite{MR1025164} used the formula to show the existence for
the entire graph solutions. To the best of our knowledge, the
monotonicity formula for the curve shortening flow with variable
mobilities is not known. We derive the weighted monotonicity identity in
a similar manner to \cite[Theorem 4.13]{MR2024995}. The key observation
in deriving the identity is the usefulness of the energy dissipation
\eqref{eq:2.14}.

A continuously differentiable function
$f:=f(x,y,t):[0,1]\times\R\times[0,\infty)\rightarrow\R$ is called
admissible if $f(0,y,t)=f(1,y,t)$ and $f_x(0,y,t)=f_x(1,y,t)$ for $y\in
\R$ and $t\geq0$. From now on, for a solution $u$ of \eqref{eq:2.13}, let
$\vec{n}=\frac{1}{\sqrt{1+|u_x|^2}}(-u_x,1)$ be an upward unit normal
vector of $\Gamma_t$, $\kappa = (\frac{u_x}{\sqrt{1+|u_x|^2}})_x$ be the
curvature of $\Gamma_t$ and $\vec{\kappa} = \kappa \vec{n}$ be the curvature vector of
$\Gamma_t$.

\begin{theorem}
 \label{thm:3.1}
 Let $(u,\alpha)$ be a solution of \eqref{eq:2.13}. Then for any $X_0\in
 \R^2$, $t_0>0$, and for any admissible
 $f:[0,1]\times\R\times[0,\infty)\rightarrow\R$,
 \begin{multline}
  \label{eq:3.5}
  \frac{d}{dt}
  \int_{\Gamma_t}
  f\rho\sigma(\alpha(t))
  =\int_{\Gamma_t}
  (f_t-\mu\sigma(\alpha(t))\Delta_{\Gamma_t}f
  +
  \mu\sigma(\alpha(t))(Df\cdot \vec{\kappa}))\rho\sigma(\alpha(t)) \\
  -
  \mu\sigma(\alpha(t))
  \int_{\Gamma_t}
  \left(
  f\rho
  \left(
  -\kappa+\frac{D\rho\cdot\vec{n}}{\rho}
  \right)^2
  \sigma(\alpha(t))
  \right)
  -
  \frac1{\gamma|\Gamma_t|}
  \int_{\Gamma_t}
  f\rho\alpha_t^2 ,
 \end{multline}
 where $\rho=\rho_{(X_0,t_0)}$ is given by \eqref{eq:3.4}.
\end{theorem}

\begin{proof}
 We first calculate
 \begin{equation}
  \begin{split}
   \frac{d}{dt}
   \int_{\Gamma_t}
   f\rho\sigma
   &=
   \int_{\Gamma_t}
   \frac{\partial}{\partial t}f\rho\sigma
   +
   \int_{\Gamma_t}
   f\frac{\partial}{\partial t}\rho\sigma
   +
   \int_{\Gamma_t}
   f\rho\sigma_\alpha\alpha_t
   +
   \int_0^1
   f\rho\sigma\frac{u_xu_{xt}}{\sqrt{1+|u_x|^2}}\,dx \\
   &=:I_1+I_2+I_3+I_4.
  \end{split}
 \end{equation}
 By integration by parts, $I_4$ is transformed into
 \begin{equation}
 \begin{split}
  I_4
  &=
  -
  \int_0^1
  \left(
   f(x,u,t)\rho(x,u,t)
  \sigma(\alpha(t))\frac{u_x}{\sqrt{1+|u_x|^2}}
  \right)_xu_t
  \,dx \\
  &=
  -
  \int_{\Gamma_t}
  \frac{\partial}{\partial x}f\rho\sigma\frac{u_x}{\sqrt{1+|u_x|^2}}
  \frac{u_t}{\sqrt{1+|u_x|^2}} \\
  &\quad
  -
  \int_{\Gamma_t}
  f\frac{\partial}{\partial x}\rho\sigma\frac{u_x}{\sqrt{1+|u_x|^2}}
  \frac{u_t}{\sqrt{1+|u_x|^2}} \\
  &\quad
  -
  \int_{\Gamma_t}
  f\rho\sigma\left(\frac{u_x}{\sqrt{1+|u_x|^2}}\right)_x
  \frac{u_t}{\sqrt{1+|u_x|^2}}.
 \end{split}
 \end{equation}
 By direct calculation of the backward heat kernel $\rho$, we have
 \begin{equation}
  \begin{split}
   \frac{\partial}{\partial t}\rho
   -
   \frac{\partial}{\partial x}\rho
   \frac{u_x}{\sqrt{1+|u_x|^2}}
   \frac{u_t}{\sqrt{1+|u_x|^2}}
   &=
   \rho_t
   +
   \rho_yu_t
   -
   (\rho_x+\rho_yu_x)
   \frac{u_x}{\sqrt{1+|u_x|^2}}
   \frac{u_t}{\sqrt{1+|u_x|^2}} \\
   &=
   \rho_t
   +
   \left(
   -\rho_x
   \frac{u_x}{\sqrt{1+|u_x|^2}}
   +\rho_y
   \frac{1}{\sqrt{1+|u_x|^2}}
   \right)
   \frac{u_t}{\sqrt{1+|u_x|^2}} \\
   &=
   \rho_t
   +
   (D\rho\cdot\vec{n})
   \frac{u_t}{\sqrt{1+|u_x|^2}},
  \end{split}
 \end{equation}
 where $\vec{n}=\frac{1}{\sqrt{1+|u_x|^2}}(-u_x,1)$.
 Similarly,
 \begin{equation}
  \frac{\partial}{\partial t}f
   -\frac{\partial}{\partial x}f
   \frac{u_x}{\sqrt{1+|u_x|^2}}
   \frac{u_t}{\sqrt{1+|u_x|^2}}
   =
   f_t
   +
   (Df\cdot\vec{n})
   \frac{u_t}{\sqrt{1+|u_x|^2}}.
 \end{equation}
 Therefore
 \begin{equation}
  \begin{split}
   I_1+I_2+I_3+I_4
   &=
   \int_{\Gamma_t}
   \left(
   f_t+(Df\cdot\vec{n})\frac{u_t}{\sqrt{1+|u_x|^2}}
   \right)\rho\sigma \\
   &\quad
   +\int_{\Gamma_t}
   f\left(
   \rho_t
   +
   \left(
   (D\rho\cdot\vec{n})
   -\rho
   \left(
   \frac{u_x}{\sqrt{1+|u_x|^2}}
   \right)_x
   \right)
   \frac{u_t}{\sqrt{1+|u_x|^2}}
   \right)\sigma \\
   &\quad
   +
   \int_{\Gamma_t}
   f\rho\sigma_\alpha\alpha_t.
  \end{split}
 \end{equation}

 Next, by equation \eqref{eq:2.13},
 \begin{equation}
  \begin{split}
   &\quad
   \left(
   (D\rho\cdot\vec{n})
   -
   \rho
   \left(
   \frac{u_x}{\sqrt{1+|u_x|^2}}
   \right)_x
   \right)
   \frac{u_t}{\sqrt{1+|u_x|^2}} \\
   &=
   \left(
   (D\rho\cdot\vec{n})
   -
   \rho \kappa
   \right)
   \cdot \mu\sigma \kappa \\
   &=
   -\mu \sigma\rho
   \left(
   \kappa^2
   -
   \frac{(D\rho\cdot\vec{n})}{\rho}
   \kappa
   \right)
   \\
   &=
   -\mu\sigma
  \left(
  \rho\left(
  -\kappa
  +
  \frac{(D\rho\cdot\vec{n})}{\rho}
  \right)^2
  -
  \frac{(D\rho\cdot\vec{n})^2}{\rho}
  +
  (D\rho\cdot\vec{n})\kappa
  \right)
   \\
  &=
   -\mu\sigma
   \left(
   \rho\left(
   -\kappa
   +
   \frac{(D\rho\cdot\vec{n})}{\rho}
   \right)^2
   -
   \frac{(D\rho\cdot\vec{n})^2}{\rho}
   +
   (D\rho\cdot\vec{\kappa})
   \right),
  \end{split}
 \end{equation}
 and
 \begin{equation}
  (Df\cdot\vec{n})\frac{u_t}{\sqrt{1+|u_x|^2}}
   =
   \mu\sigma(Df\cdot\vec{\kappa}).
 \end{equation}
 Again, we use equation \eqref{eq:2.13} and
 \begin{equation}
  \begin{split}
   I_1+I_2+I_3+I_4
   &=
   \int_{\Gamma_t}
   (f_t+\mu\sigma(Df\cdot\vec{\kappa}))\rho\sigma
   +
   \int_{\Gamma_t}
   f
   \left(
   \rho_t
   +
   \mu\sigma
   \frac{(D\rho\cdot\vec{n})^2}{\rho}
   -
   \mu\sigma
   (D\rho\cdot\vec{\kappa})
   \right)
   \sigma \\
   &\quad
   -\mu\sigma
   \int_{\Gamma_t}
   f\rho\left(
   -\kappa
   +
   \frac{(D\rho\cdot\vec{n})}{\rho}
   \right)^2\sigma
   -\frac{1}{\gamma|\Gamma_t|}
   \int_{\Gamma_t}
   f\rho\alpha_t^2.
  \end{split}
 \end{equation}

 By Gauss' divergence formula and assumption $f(0,y,t)=f(1,y,t)=0$, we
 have
 \begin{equation*}
  \int_{\Gamma_t}
   \Div_{\Gamma_t}(fD\rho)
   =
   -
   \int_{\Gamma_t}
   f(D\rho\cdot\vec{\kappa}).
 \end{equation*}
 Here,
 \begin{equation}
  \begin{split}
   \Div_{\Gamma_t}(fD\rho)
   &=
   \frac{1}{\sqrt{1+|u_x|^2}}\frac{\partial}{\partial x}
   (f(x,u,t)(\rho_x(x,u,t),\rho_y(x,u,t)))
   \cdot
   \frac{(1,u_x)}{\sqrt{1+|u_x|^2}} \\
   &=
   \frac{f}{1+|u_x|^2}
   \left(
   \begin{pmatrix}
    1 & u_x \\
    u_x & |u_x|^2
   \end{pmatrix}
   :
   D^2\rho
   \right)
   +
   \frac{1}{\sqrt{1+|u_x|^2}}(\rho_x+\rho_yu_x)
   \left(\frac{1}{\sqrt{1+|u_x|^2}}\frac{\partial}{\partial x}\right)f \\
   &=
   f(I-\vec{n}\otimes\vec{n}):D^2\rho
   +
   \left(\frac{1}{\sqrt{1+|u_x|^2}}\frac{\partial}{\partial x}\right)\rho
   \left(\frac{1}{\sqrt{1+|u_x|^2}}\frac{\partial}{\partial x}\right)f.
  \end{split}
 \end{equation}
 With $f$ admissible, we obtain by integration by parts
 \begin{equation}
  \begin{split}
   \int_{\Gamma_t}
   \left(\frac{1}{\sqrt{1+|u_x|^2}}\frac{\partial}{\partial x}\rho\right)
   \left(\frac{1}{\sqrt{1+|u_x|^2}}\frac{\partial}{\partial x}f\right)
   &=
   \int_{0}^1
   \frac{\partial}{\partial x}\rho
   \left(\frac{1}{\sqrt{1+|u_x|^2}}\frac{\partial}{\partial x}f\right)
   \,dx \\
   &=
   -\int_{0}^1
   \rho
   \left(
   \frac{1}{\sqrt{1+|u_x|^2}}\frac{\partial}{\partial x}f
   \right)_x
   \,dx \\
   &=
   -\int_{\Gamma_t}
   \rho\Delta_{\Gamma_t}f.
  \end{split}
 \end{equation}
 Therefore, by \eqref{eq:3.2} we obtain
 \begin{equation}
  \begin{split}
   \frac{d}{dt}
   \int_{\Gamma_t}
   f\rho\sigma
   &=
   \int_{\Gamma_t}
   (f_t
   -
   \mu\sigma\Delta_{\Gamma_t}f
   +
   \mu\sigma(Df\cdot\vec{\kappa}))
   \rho\sigma \\
   &\quad
   -\mu\sigma
   \int_{\Gamma_t}
   f\rho\left(
   -\kappa
   +
   \frac{(D\rho\cdot\vec{n})}{\rho}
   \right)^2\sigma
   -\frac{1}{\gamma|\Gamma_t|}
   \int_{\Gamma_t}
   f\rho\alpha_t^2.
  \end{split}
 \end{equation}
\end{proof}


\begin{remark}
 Equality \eqref{eq:3.5} also holds when $\Gamma_t$ is not a graph.
 A key relation in proving \eqref{eq:3.5} is
 \begin{equation}
  \label{eq:3.10}
  \frac{d}{dt} \int_{\Gamma_t} F
   =
   \int_{\Gamma_t} \{ (\nabla F - F \vec{\kappa}) \cdot \vec{v}_n + F_t \}.
 \end{equation}
 for any smooth function $F: \mathbb{R}^2 \times [0,\infty) \to
 \mathbb{R}$, where $\vec{v}_n$ and $\vec{\kappa}$ denote the normal
 velocity vector and the curvature vector of $\Gamma_t$,
 respectively. Indeed, the relation \eqref{eq:3.10} also holds for a
 smooth Jordan curve $\Gamma_t$ (see \cite[Proposition 4.6 and Theorem
 4.13]{MR2024995}).
\end{remark}

On the proof of Theorem \ref{thm:3.1}, we only use the smoothness of the
energy density $\sigma$. If we assume the positivity~\eqref{eq:1.A1} and
the non-negativity of the admissible function $f$, we obtain the
weighted monotonicity formula.

\begin{corollary}
 Let $(u,\alpha)$ be a solution of \eqref{eq:2.13} and let
 $f:[0,1]\times\R\times[0,\infty)\rightarrow[0,\infty)$ be a
 non-negative admissible function. Then, under assumption
 \eqref{eq:1.A1}, we obtain
 \begin{equation}
  \label{eq:3.6}
  \frac{d}{dt}
  \int_{\Gamma_t}
  f\rho\sigma(\alpha(t))
  \leq
  \int_{\Gamma_t}
  (f_t-\mu\sigma(\alpha(t))\Delta_{\Gamma_t}f
  +
  \mu\sigma(\alpha(t))(Df\cdot\vec{\kappa}))\rho\sigma(\alpha(t)),
 \end{equation}
 where $\rho=\rho_{(X_0,t_0)}$ is given by \eqref{eq:3.4}.
\end{corollary}

\section{Gradient estimates and existence of solutions}

In this section, we first obtain the a priori gradient estimates by
applying the area element $\sqrt{1+|u_x|^2}$ to the admissible function
in the weighted monotonicity formula, obtained in previous section. Note
that the area element is the non-negative admissible function and the
integrand of the right hand side of \eqref{eq:3.6} is non-positive.
Next, we prove the existence of classical solutions for
\eqref{eq:2.13} from the a priori gradient estimates.

\begin{lemma}
 \label{lem:4.1}
 Let $(u,\alpha)$ be a solution of \eqref{eq:2.13} and let
 $v:=\sqrt{1+|u_x|^2}$. Then
 \begin{equation}
  \label{eq:4.2}
  v_t
   -
   \mu\sigma\Delta_{\Gamma_t}v
   +
   \mu\sigma(Dv\cdot\vec{\kappa})
   =
   -
   \mu\sigma v\kappa^2
   -
   2\mu\sigma
   \frac{v_x^2}{v^3}.
 \end{equation}
\end{lemma}

\begin{proof}
 Taking a derivative of \eqref{eq:2.13} with respect to $x$, we obtain
 \begin{equation*}
   u_{tx}
   =
   \mu
   \sigma(\alpha)
   \left(
    v_x \kappa
    +
    v\kappa_x
   \right).
 \end{equation*}
 Multiplying $u_x/v$ and using the relation $vv_t=u_xu_{xt}$, we have
 \begin{equation}
  \label{eq:4.3}
  v_t
   =
   \mu
   \sigma(\alpha)
   \left(
    \frac{u_xv_x}{v}\kappa
    +
    u_x \kappa_x
   \right).
 \end{equation}
 Next, we manipulate the curvature $\kappa$ as
 \begin{equation*}
  \kappa
   =
   \left(\frac{u_x}{v}\right)_x
   =
   \left(
    \frac{u_{xx}}{v}-\frac{u_x^2u_{xx}}{v^3}
   \right)
   =
   \frac{u_{xx}}{v^3}(v^2-u_x^2)
   =
   \frac{u_{xx}}{v^3}.
 \end{equation*}
 Let $\partial_{\Gamma_t}=\frac{1}{v}\partial_x$ be the derivative along
 $\Gamma_t$. Then,
 $\Delta_{\Gamma_t}=\partial_{\Gamma_t}^2$ and
 \begin{equation*}
  \partial_{\Gamma_t}v
   =
   \frac{1}{v}v_x
   =
   \frac{1}{v^2}u_xu_{xx}
   =
   v^2\frac{u_x}{v}\kappa.
 \end{equation*}
 Therefore
 \begin{equation}
  \label{eq:4.4}
   \begin{split}
    \Delta_{\Gamma_t}v
    =
    \frac1v
    (\partial_{\Gamma_t}v)_x
    &=
    \frac{2v_xu_x \kappa}{v}
    +
    v\kappa^2
    +
    u_x\kappa_x \\
    &=
    \frac{2v^2_x}{v^3}
    +
    v\kappa^2
    +
    u_x\kappa_x.
   \end{split}
\end{equation}
 Since
 \begin{equation}
  \label{eq:4.1}
   Dv\cdot\vec{\kappa}
   =
   v_x\left(-\kappa\frac{u_x}{v}\right),
 \end{equation}
 we obtain \eqref{eq:4.2} by direct substitution of \eqref{eq:4.3},
 \eqref{eq:4.4}, and \eqref{eq:4.1}.
\end{proof}


\begin{theorem}
 \label{thm:4.2}
 Let $(u,\alpha)$ be a solution of \eqref{eq:2.13} and let $v:=\sqrt{1+u_x^2}$.
 Assume \eqref{eq:1.A1}. Then, for all $0<x_0<1$ and
 $t_0>0$,
 \begin{equation}
 \label{eq:4.9}
  v(x_0,t_0)
   \leq
   \frac{\sigma(\alpha(0))}{\Cr{const:1.1}}
   \sup_{0<x<1}v^2(x,0).
 \end{equation}
\end{theorem}

\begin{proof}
 Put $X_0=(x_0,u(x_0,t_0))$ and consider the backward heat kernel
 $\rho=\rho_{(X_0,t_0)}$. Then, Theorem \ref{thm:3.1} with $f=v$ and
 Lemma \ref{lem:4.1} imply
 \begin{equation}
  \frac{d}{dt}
   \int_{\Gamma_t}
   v\rho\sigma(\alpha(t))
   \leq
   -
   \int_{\Gamma_t}
   \left(
    \mu\sigma v\kappa^2
    +
    2\mu\sigma
    \frac{v_x^2}{v^3}
   \right)
   \rho\sigma(\alpha(t))
   \leq
   0
 \end{equation}
 for $0<t<t_0$. Here we use the non-negativity of $\sigma$. Thus
 \begin{equation}
  \label{eq:4.5}
  \begin{split}
   \int_{\Gamma_t}
   v(x,t)\rho(X,t)\sigma(\alpha(t))
   &\leq
   \int_{\Gamma_0}
   v(x,0)\rho(X,0)\sigma(\alpha(0)) \\
   &\leq
   \sigma(\alpha(0))
   \sup_{0<x<1}v(x,0)
   \int_0^1\rho((x,u(x,0)),0)v(x,0)\,dx \\
   &\leq
   \sigma(\alpha(0))
   \sup_{0<x<1}v^2(x,0).
  \end{split}
 \end{equation}
 Taking a limit $t\uparrow t_0$ on \eqref{eq:4.5} and Assumption
 \eqref{eq:1.A1}, we have
 \begin{equation}
  \Cr{const:1.1}v(x_0,t_0)
   \leq
   \sigma(\alpha(t_0))v(x_0,t_0)
   \leq
   \sigma(\alpha(0))
   \sup_{0<x<1}v^2(x,0).
 \end{equation}
\end{proof}

\begin{lemma}
 \label{lem:4.2} Let $(u,\alpha)$ be a solution of
 \eqref{eq:2.13}. Assume \eqref{eq:1.A2}. Then, for all $t_0>0$
 \begin{equation}
  \label{eq:4.6}
 |\alpha(t_0)|
   \leq
   |\alpha(0)|.
 \end{equation}
\end{lemma}

\begin{proof}
 Multiplying the equation \eqref{eq:2.13} by $\alpha$ and using
 \eqref{eq:1.A2} imply
 \begin{equation}
  \frac12
   (\alpha^2)_t
   =
   -\gamma \alpha\sigma_\alpha(\alpha)
   |\Gamma_t|
   \leq
   0.
 \end{equation}
 Integrating the above inequality on $0\leq t\leq t_0$, we have
 \eqref{eq:4.6}.
\end{proof}



In a similar manner to the arguments in \cite{MR0762825}, the following
holds:
\begin{lemma}
 Let $(u,\alpha)$ be a solution of \eqref{eq:2.13}. Assume
 \eqref{eq:1.A1}. Then, for all $0<x_0<1$ and $t_0>0$,
 \begin{equation}
  \label{eq:4.8}
  |u(x,t)|
   \leq
   \sup_{0<x<1}|u(x,0)|.
 \end{equation}
\end{lemma}

\begin{proof}
 Let $M:=\sup_{0<x<1}u(x,0)$ and assume that there is a point $(x_0,t_0)\in
 (0,1)\times(0,\infty)$ such that $u$ takes maximum $M_1$, which is
 greater than $M$, at the point $(x_0,t_0)$. At this point, we have
 \begin{equation}
  u(x_0,t_0)=M_1>M,\quad
  u_x(x_0,t_0)=0,\quad
   u_{xx}(x_0,t_0)\leq0,\ \text{and}\
   u_t(x_0,t_0)\geq0.
 \end{equation}
 Let us define
 \begin{equation}
  w(x,t):=u(x,t)+\frac{M_1-M}{2}(x-x_0)^2.
 \end{equation}
 Then
 \begin{equation}
  w(x,0)=u(x,0)+\frac{M_1-M}{2}(x-x_0)^2
   \leq M+\frac{M_1-M}{2}
   < M_1,\quad
   w(x_0,t_0)=M_1.
 \end{equation}
 Therefore, the maximum point $(x_1,t_1)$ of $w$ is in the interior of
 $(0,1)\times(0,\infty)$. From equation \eqref{eq:2.13}, we obtain
 a differential inequality
 \begin{equation}
  \label{eq:4.7}
  w_t
   =
   u_t
   =
   \mu\sigma(\alpha(t))
   \frac{u_{xx}}{v^2}
   <
   \mu\sigma(\alpha(t))
   \frac{w_{xx}}{v^2}.
 \end{equation}
 At point $(x_1,t_1)$, the left hand side of \eqref{eq:4.7} is
 non-negative but the right hand side of \eqref{eq:4.7} is non-positive.
 This is a contradiction, and therefore, there is no interior point
 $(x_0,t_0)\in(0,1)\times(0,\infty)$ such that $u$ takes a maximum at
 $(x_0,t_0)$. Similarly, $u$ does not take minimum at any interior point
 of $(0,1)\times(0,\infty)$; thus we obtain \eqref{eq:4.8}.
\end{proof}

We recall $\mathbb{T} = \mathbb{R} /\mathbb{Z}$. Let $Q_T := \mathbb{T}
\times (0,T)$, and $Q_T ^\varepsilon := \mathbb{T} \times
(\varepsilon,T)$ be parabolic cylinders for
$0<\varepsilon<T<\infty$. Using the $L^\infty$-estimates and the
gradient estimates, we obtain the time global existence
theorem:

\begin{theorem}
 \label{thm:4.5}
 Assume that $u_0 $ is a Lipschitz function on $\mathbb{T}$ with a
 Lipschitz constant $M>0$, $\beta \in (0,1)$, $\alpha _0 \in \mathbb{R}$
 and $\sigma \in C^1 (\mathbb{R})$ satisfies \eqref{eq:1.A1} and
 \eqref{eq:1.A2}.  Moreover, there exists $L>0$ such that $|\sigma
 _\alpha (a) - \sigma_\alpha (b) | \leq L |a-b| $ for any $a,b \in
 \mathbb{R}$.  Then, for any $0<\varepsilon <T<\infty$, there exists a
 unique solution $(u,\alpha) \in C (\overline{Q_T} ) \cap C^{2,\beta}
 (Q_T ^\varepsilon) \times C([0,T)) \cap C^{1,1} ((\varepsilon ,T))$ of
 \eqref{eq:2.13} with $(u (\cdot ,0),\alpha(0)
 )=(u_0,\alpha_0)$. Furthermore, we have
  \begin{equation}
 \label{eq:4.27}
  \| u \|_{C^{2,\beta} (Q_T ^\varepsilon )} \leq \Cl{const:4.5},
 \end{equation}
 where $\Cr{const:4.5} >0$ depends only on $\gamma$, $\mu$, $\varepsilon $, $L$, $M$,
 $\Cr{const:1.1}$, and $\sigma(\alpha(0))$.
\end{theorem}

\begin{proof}
Set $T>0$ and $0<\beta< 1$ and $X:= C^{1, \beta} (Q_T)$.  First, we
assume $u_0 \in C^{2,\beta} (\Omega)$.  Let $w \in X$. Then, $f_w (t) :=
\int _0 ^1 \sqrt{1+|w_x (x,t)| ^2} \, dx$ is continuous and bounded in
$[0,T]$. In addition, the function $g_w (\alpha,t) :=-\gamma
\sigma_\alpha (\alpha) f_w (t) $ is continuous and $|g_w (\alpha,t) -
g_w (\beta,t) | \leq \gamma L (1+\| w \|_X) |\alpha-\beta| $ for any
$\alpha,\beta \in \mathbb{R}$ and $t \in [0,T]$.  Therefore, there
exists a unique solution $\alpha_w (t) $ of
 \begin{equation}
 \label{eq:4.16}
 \left\{
  \begin{aligned}
   (\alpha_w)_t (t)
   &=
   g_w (\alpha_w (t),t),
   &
   \quad
   &t \in (0,T), \\
    \alpha_w (0)&=
    \alpha_0.
  \end{aligned}
 \right.
\end{equation}
 With the same argument as in Lemma \ref{lem:4.2}, we have
 $|\alpha_w(t)|\leq|\alpha_0|$ for $t>0$ from assumption
 \eqref{eq:1.A2}. Thus
\begin{equation*}
 \left| \frac{d}{dt} \alpha_w (t) \right|
 \leq \gamma L (1+\|w\|_X) |\alpha_w (t)|
 \leq \gamma L (1+\|w\|_X) |\alpha_0|, \quad t \in (0,T),
\end{equation*}
and
\begin{equation}
 \label{eq:4.19}
 \left| \frac{d}{dt} \sigma (\alpha_w (t) )\right|
 \leq \gamma L^2 (1+\|w\|_X) |\alpha_0|^2, \quad t \in (0,T),
\end{equation}
where \eqref{eq:1.A2} is used.
Next, we consider the following linearized equation:
\begin{equation}
 \label{eq:4.17}
 \left\{
  \begin{aligned}
   u_t
   &=
   \mu\sigma(\alpha_w)
   \frac{u_{xx}}{1+|w_x|^2},
   &
   \quad
   &x\in \mathbb{T},\ t>0, \\
    u(x,0)&=
    u_0(x),&\quad
    &x\in \mathbb{T}.
  \end{aligned}
 \right.
\end{equation}
Note that $\left\| \frac{\mu\sigma(\alpha_w)}{1+|w_x|^2} \right\| _\infty$ is bounded in $Q_T$ and
\eqref{eq:4.17} is uniformly parabolic in $Q_T$. In addition, we compute
\begin{equation}
 \label{eq:4.20}
  \begin{split}
   \left| \frac{1}{1+|w_x(x,t) |^2}
   -
   \frac{1}{1+|w_x(y,s) |^2} \right|
   &\leq
   \frac{|w_x(x,t) |+|w_x(y,t) |}{(1+|w_x(x,t) |^2)(1+|w_x(y,t) |^2)}
   |w_x (x,t) -w_x (y,s)| \\
   &\leq
   |w_x (x,t) -w_x (y,s)|
  \end{split}
\end{equation}
for any $ (x,t),(y,s) \in \mathbb{T} \times [0,T]$. Therefore,
\eqref{eq:4.19} and \eqref{eq:4.20} imply
\begin{equation}
 \label{eq:4.21}
 \left\|
  \frac{\mu\sigma(\alpha_w)}{1+|w_x|^2}
 \right\|_{C^\beta (Q_T)}
 \leq
 \mu
 (\sup_{|\alpha|\leq|\alpha_0|}
 |\sigma(\alpha)|
 (1+\|w\|_X)
 + \gamma L^2 (1+\|w\|_X) |\alpha_0|^2)
\end{equation}
for any $w \in X$. Thus there exists a unique solution
$u_w \in C ^{2,\beta} (Q_T)$ of \eqref{eq:4.17} with
\begin{equation}
\label{eq:4.18}
\| u_w \|_{C^{2,\beta} (Q_T)} \leq \Cl{const:4.1},
\end{equation}
where $\Cr{const:4.1}>0$ depends only on $\gamma$, $\mu$, $\| w\|_X$,
$L$, $|\alpha_0|$, and $\| u_0 \|_{C^{2,\beta} (\mathbb{T})}$.
Next, we define $A: X\to X$ by $Aw = u_w$.
We remark that $A$ is a continuous and compact operator. Set
\[
S:= \{ u \in X \; | \; u = \eta A u \ \text{in} \ X, \
\text{for some} \ \eta \in [0,1] \}.
\]
Next, we show that $S$ is bounded in $X$. For any $u \in S$, we have
\begin{equation}
 \label{eq:4.22}
 \left\{
  \begin{aligned}
   \frac{u_t}{\sqrt{1+|u_x|^2}}
   &=
   \mu\sigma(\alpha)
   \left(
   \frac{u_x}{\sqrt{1+|u_x|^2}}
   \right)_x,&
   \quad
   &x\in\mathbb{T},\ t>0, \\
   \alpha_t
   &=
   -
   \gamma
   \sigma_\alpha(\alpha)
   |\Gamma_t|, &
   \quad
   &t>0,
   \\
    u(x,0)&=
    \eta u_0(x),&\quad
    &x\in \mathbb{T},\\
    \alpha(0)&=
   \alpha_0,
  \end{aligned}
 \right.
\end{equation}
for some $\eta \in [0,1]$. Here $|\Gamma _t |= \int _0 ^1 \sqrt{1+|u_x (x,t)| ^2 } \, dx$.
The gradient estimate \eqref{eq:4.9} implies
 \begin{equation}
 \label{eq:4.23}
  \sup _{Q_T} |u_x |
   \leq
   \frac{\sigma(\alpha(0))}{\Cr{const:1.1}}
   \sup_{0<x<1}(1+\eta ^2 | (u_0)_x | ^2 ).
 \end{equation}
 By \eqref{eq:4.8}, \eqref{eq:4.23} and the interior Schauder estimates
 (cf. \cite[Theorem 6.2.1]{MR0241822}) we have
  \begin{equation}
 \label{eq:4.24}
   \|u_x \|_{C^\beta (Q_T)}
   \leq
   \Cl{const:4.2},
 \end{equation}
 where $\Cr{const:4.2} >0$ depends only on $\Cr{const:1.1}$,
 $\sigma(\alpha(0))$, $\sup_{0<x<1}  |u _0 (x) |  $, and
 $\sup_{0<x<1}  | (u_0)_x (x) | $.
 Therefore, by an argument similar to \eqref{eq:4.18}, we obtain
 \begin{equation}
 \label{eq:4.25}
 \|u \|_X \leq \| u \|_{C^{2,\beta} (Q_T)} \leq \Cl{const:4.3},
 \end{equation}
 where $\Cr{const:4.3} >0$ depends only on $\Cr{const:1.1}$,
 $\sigma(\alpha(0))$, $\| u_0 \|_{C^{2,\beta} (Q_T)}$.  Hence, $S$ is
 bounded in $X$ and the Leray-Schauder fixed point theorem implies that
 there exists a solution $(u,\alpha) \in C^{2,\beta} (Q_T)\times C
 ^{1,1} (0,T)$ of \eqref{eq:2.13}.

Next, we consider the case when $u_0 $ is a Lipschitz function with
Lipschitz constant $M>0$. Set $\varepsilon \in (0,T)$.  Let $\{ u ^i
_0\} _{i=1}^\infty$ be a family of smooth functions such that $u_0 ^i$
converges uniformly to $u_0$ on $\mathbb{T}$.  Then, \eqref{eq:4.9}
implies
\[
  \sup _{Q_T} |u_x ^i |
   \leq
   \frac{\sigma(\alpha^i (0))}{\Cr{const:1.1}}
   \sup_{0<x<1}(1+ M ^2 ), \quad i\geq 1,
\]
where $(u^i,\alpha ^i)$ is a solution of \eqref{eq:2.13} with
$(u^i (\cdot,0),\alpha^i (0)) =(u_0 ^i,\alpha_0)$.
Using a similar argument as for \eqref{eq:4.24} and \eqref{eq:4.25},
along with the interior Schauder estimates, we have
 \begin{equation}
 \label{eq:4.26}
 \sup _ i \| u^i \|_{C^{2,\beta} (Q_T ^\varepsilon )} \leq \Cl{const:4.4},
 \end{equation}
 where $\Cr{const:4.4} >0$ depends only on $\gamma$, $\mu$, $\varepsilon
 $, $L$, $M$, $\Cr{const:1.1}$, and $\sigma(\alpha(0))$.  Therefore, by
 taking the subsequence, $(u^i,\alpha^i)$ converges to a solution
 $(u,\alpha)$ in $Q_T ^\varepsilon$ with \eqref{eq:4.27}. Thus, from the
 diagonal arguments, we obtain a solution $(u,\alpha) \in C
 (\overline{Q_T} ) \cap C^{2,\beta} (Q_T ^\varepsilon) \times C([0,T))
 \cap C^{1,1} (\varepsilon ,T)$ of \eqref{eq:2.13} with $(u (\cdot
 ,0),\alpha(0) )=(u_0,\alpha_0)$. Uniqueness is obvious from the
 comparison principle, and thereby, we prove Theorem \ref{thm:4.5}.
\end{proof}

We remark that assumption \eqref{eq:1.A1} is not a necessary condition
to obtain the gradient estimate. For example, consider
\begin{equation}
 \sigma(\alpha)=\frac{1}{2}\alpha^2.
\end{equation}
Then assumption \eqref{eq:1.A1} does not hold so we cannot use Theorem
\ref{thm:4.2} directly. However, we may write $\alpha(t)$ explicitly as
\begin{equation}
 \label{eq:5.3}
  \alpha(t)
  =
  \alpha(0)
  \exp
  \left(
   -
   \int_0^t|\Gamma_\tau|\,d\tau
  \right)
\end{equation}
so we obtain
\begin{equation}
 \begin{split}
  v(t_0,x_0)
  &\leq
  \exp
  \left(
  2\int_0^{t_0}
  |\Gamma_\tau|\,d\tau
  \right)
  \sup_{0<x<1}v^2(x,0) \\
  &\leq
  \exp
  \left(
  2t_0|\Gamma_0|
  \right)
  \sup_{0<x<1}v^2(x,0),
 \end{split}
\end{equation}
provided $\alpha(0)\neq0$.

From \eqref{eq:5.3} and $|\Gamma_t|\geq1$ for $t>0$, we have
\begin{equation}
 \label{eq:4.31}
 |\alpha(t)|
  \leq
  |\alpha(0)|
  \exp
  \left(
   -t
  \right).
\end{equation}
Hence the misorientation $\alpha(t)$ goes to $0$ exponentially as
$t\rightarrow\infty$.

\section{Asymptotics of solutions}

In regard to Theorem \ref{thm:4.5}, we can take $T=\infty$ and show the
existence of a unique time global solution of \eqref{eq:2.13}. In this
section, we study large time asymptotic behavior for the
solution. Without loss of generality, we assume that the initial data
$u_0$ is sufficiently smooth by the Schauder estimates.

\begin{theorem}
 \label{thm:5.1}
 Let $u_0:\mathbb{T}\rightarrow\R$, $\alpha_0\in\R$ and assume the same
 assumption as for Theorem \ref{thm:4.5}. Let $(u,\alpha)$ be a time
 global solution of \eqref{eq:2.13}. Then, there exists a constant
 $u_\infty$ such that $\| u_\infty -u \|_{C^2 (\mathbb{T})}$ goes to $0$
 exponentially. In addition, the curvature $\kappa$ also goes to $0$
 uniformly and exponentially on $\mathbb{T}$.
\end{theorem}

 To prove Theorem \ref{thm:5.1}, we first derive the energy dissipation
 estimates for \eqref{eq:2.13}. In fact, the estimates are obvious from
 the derivation of equation \eqref{eq:2.13}.


\begin{proposition}
 \label{prop:5.1}
 Let $(u,\alpha)$ be a solution of \eqref{eq:2.13}. Then
 \begin{equation}
  \label{eq:5.1}
  \frac{d}{dt}|\Gamma_t|
   +
   \mu\sigma(\alpha(t))
   \int_{\Gamma_t}
   \kappa^2
   =0,
 \end{equation}
 where $\kappa=\left(\frac{u_x}{\sqrt{1+u_x^2}}\right)_x$.
\end{proposition}

\begin{proof}
 Taking the time derivative to $|\Gamma_t|$ and integrating by parts, we
 obtain
 \begin{equation*}
  \begin{split}
   \frac{d}{dt}|\Gamma_t|
   &=
   \int_{0}^1\frac{u_xu_{xt}}{\sqrt{1+u_x^2}}\,dx \\
   &=
   -\int_{0}^1
   \left(
   \frac{u_x}{\sqrt{1+u_x^2}}
   \right)_x
   u_t\,dx \\
   &=
   -\mu\sigma(\alpha(t))\int_{0}^1
   \left(
   \frac{u_x}{\sqrt{1+u_x^2}}
   \right)_x^2
   \sqrt{1+u_x^2}\,dx
   =
   -\mu\sigma(\alpha(t))\int_{\Gamma_t}
   \kappa^2.
  \end{split}
\end{equation*}
\end{proof}

Since the second term of left hand side of \eqref{eq:5.1} is
non-negative, $\frac{d}{dt}|\Gamma_t|$ has to be non-positive, and hence
we have
\begin{corollary}
 Let $(u,\alpha)$ be a solution of \eqref{eq:2.13}. Assume
 $\sigma\geq0$. Then $|\Gamma_t|\leq |\Gamma_0|$ for $t>0$.
\end{corollary}

From Proposition \ref{prop:5.1}, $\kappa^2$ is integrable on
$(0,1)\times(0,\infty)$. Hence,
\begin{corollary}
 Let $(u,\alpha)$ be a time global solution of \eqref{eq:2.13}. Assume
 \eqref{eq:1.A1}. Then, there is a sequence
 $\{t_j\}_{j=1}^\infty$ such that $t_j\rightarrow\infty$ and
 $\kappa(x, t_j)\rightarrow0$ almost all $x\in(0,1)$ as
 $j\rightarrow\infty$.
\end{corollary}

%
%
%
%

We derive more explicit decay estimates via the energy methods. Note that
if $(u,\alpha)$ is a classical solution of \eqref{eq:2.13}, then
\begin{equation}
 u_t
  =
  \frac{\mu\sigma(\alpha)}{1+|u_x|^2}u_{xx}.
\end{equation}
Taking the derivative with respect to $x$, we obtain
\begin{equation}
 \label{eq:5.10}
  \begin{split}
   u_{xt}
   &=
   \frac{\mu\sigma(\alpha)}{1+|u_x|^2}u_{xxx}
   -
   \frac{2\mu\sigma(\alpha)}{(1+|u_x|^2)^2}u_xu^2_{xx}, \\
   u_{xxt}
   &=
   \frac{\mu\sigma(\alpha)}{1+|u_x|^2}u_{xxxx}
   -
   \frac{6\mu\sigma(\alpha)}{(1+|u_x|^2)^2}u_xu_{xx}u_{xxx}
   -
   \frac{2\mu\sigma(\alpha)}{(1+|u_x|^2)^2}u^3_{xx}
   +
   \frac{8\mu\sigma(\alpha)}{(1+|u_x|^2)^3}u^2_{x}u_{xx}^3, \\
   u_{xxxt}
   &=
   \frac{\mu\sigma(\alpha)}{1+|u_x|^2}u_{xxxxx}
   -
   \frac{8\mu\sigma(\alpha)}{(1+|u_x|^2)^2}u_xu_{xx}u_{xxxx}
   -
   \frac{12\mu\sigma(\alpha)}{(1+|u_x|^2)^2}u^2_{xx}u_{xxx}
   -
   \frac{6\mu\sigma(\alpha)}{(1+|u_x|^2)^2}u_{x}u^2_{xxx} \\
   &\quad
   +
   \frac{48\mu\sigma(\alpha)}{(1+|u_x|^2)^3}u^2_xu^2_{xx}u_{xxx}
   +
   \frac{24\mu\sigma(\alpha)}{(1+|u_x|^2)^3}u_xu^4_{xx}
   -\frac{48\mu\sigma(\alpha)}{(1+|u_x|^2)^4}u^3_{x}u_{xx}^4.
  \end{split}
\end{equation}

\begin{proposition}
 Let $(u,\alpha)$ be a classical solution of \eqref{eq:2.13}. Then
 there exists $\Cl{const:5.1}>0$ such that
 \begin{equation}
  \label{eq:5.11}
   \int_0^1
   |u_x(x,t)|^2\,dx
   \leq
   e^{-\Cr{const:5.1}t}
   \int_0^1
  |u_{0x}(x)|^2\,dx
 \end{equation}
 for $t>0$.
\end{proposition}

\begin{proof}
 Taking the derivative of the left hand side of \eqref{eq:5.11} and then
 integrating by parts, we have
 \begin{equation}
  \begin{split}
   \frac{d}{dt}
   \int_0^1
   |u_x(x,t)|^2\,dx
   &=
   2\int_0^1
   u_x(x,t)u_{xt}(x,t)\,dx \\
   &=
   2\int_0^1
   \frac{\mu\sigma(\alpha)}{1+|u_x|^2}u_xu_{xxx}\,dx
   -
   4\int_0^1
   \frac{\mu\sigma(\alpha)}{(1+|u_x|^2)^2}u^2_xu^2_{xx}\,dx \\
   &=
   -2\int_0^1
   \frac{\mu\sigma(\alpha)}{1+|u_x|^2}u^2_{xx}\,dx.
  \end{split}
 \end{equation}
 Using Assumption \eqref{eq:1.A1}, Theorem \ref{thm:4.2} and the
 Poincar\'e inequality, we obtain
 \begin{equation}
  -2
   \int_0^1
   \frac{\mu\sigma(\alpha)}{1+|u_x|^2}u^2_{xx}\,dx
   \leq
   -
   \Cr{const:5.1}
   \int_0^1
   |u_x(t,x)|^2\,dx,
 \end{equation}
 where $\Cr{const:5.1}>0$ is a positive constant depending only on
 $\mu$, $\Cr{const:1.1}$. $\sigma(\alpha_0)$,
 $\sup_{x\in\mathbb{T}}(1+u_{0x}^2(x))$. By the Gronwall inequality, we
 obtain \eqref{eq:5.11}.
\end{proof}

\begin{proposition}
 \label{prop:5.6}
 Let $(u,\alpha)$ be a classical solution of \eqref{eq:2.13}. Then
 there exists $\Cl{const:5.2}>0$ such that
 \begin{equation}
  \label{eq:5.12}
   \int_0^1
   |u_{xx}(x,t)|^2\,dx
   \leq
   e^{-\Cr{const:5.2}t}
   \int_0^1
  |u_{0xx}(x)|^2\,dx
 \end{equation}
 for $t>0$.
\end{proposition}

\begin{proof}
 Taking the derivative of the left hand side of \eqref{eq:5.12} and then
 integrating by parts, we have
 \begin{equation}
  \begin{split}
   \frac{d}{dt}
   \int_0^1
   |u_{xx}(x,t)|^2\,dx
   &=
   2\int_0^1
   u_{xx}(x,t)u_{xxt}(x,t)\,dx \\
   &=
   2\int_0^1
   \frac{\mu\sigma(\alpha)}{1+|u_x|^2}u_{xx}u_{xxxx}\,dx
   -
   12\int_0^1
   \frac{\mu\sigma(\alpha)}{(1+|u_x|^2)^2}u_xu^2_{xx}u_{xxx}\,dx \\
   &\quad
   -
   4\int_0^1\frac{\mu\sigma(\alpha)}{(1+|u_x|^2)^2}u^4_{xx} \,dx
   +
   16\int_0^1
   \frac{\mu\sigma(\alpha)}{(1+|u_x|^2)^3}u^2_{x}u_{xx}^4\,dx \\
   &=
   -2\int_0^1
   \frac{\mu\sigma(\alpha)}{1+|u_x|^2}u^2_{xxx}\,dx
   -
   8\int_0^1
   \frac{\mu\sigma(\alpha)}{(1+|u_x|^2)^2}u_xu^2_{xx}u_{xxx}\,dx \\
   &\quad
   -
   4\int_0^1\frac{\mu\sigma(\alpha)}{(1+|u_x|^2)^2}u^4_{xx} \,dx
   +
   16\int_0^1
   \frac{\mu\sigma(\alpha)}{(1+|u_x|^2)^3}u^2_{x}u_{xx}^4\,dx.
  \end{split}
 \end{equation}
 By the Young inequality,
 \begin{equation*}
  8
   \left|
    \int_0^1
    \frac{\mu\sigma(\alpha)}{(1+|u_x|^2)^2}u_xu^2_{xx}u_{xxx}\,dx
   \right|
   \leq
   \int_0^1
   \frac{\mu\sigma(\alpha)}{(1+|u_x|^2)}u^2_{xxx}\,dx
   +
   16
   \int_0^1
    \frac{\mu\sigma(\alpha)}{(1+|u_x|^2)^6}u^2_xu^4_{xx}\,dx,
 \end{equation*}
 hence
 \begin{equation*}
  \begin{split}
   \frac{d}{dt}
   \int_0^1
   |u_{xx}(x,t)|^2\,dx
   \leq
   -\int_0^1
   \frac{\mu\sigma(\alpha)}{1+|u_x|^2}u^2_{xxx}\,dx
   +
   \Cr{const:5.3}\int_0^1 u_x^2\,dx,
  \end{split}
 \end{equation*}
 where $\Cl{const:5.3}>0$ depends only on $\mu$,
 $\sigma(\alpha(0))$ and $\Cr{const:4.5}$. Using Assumption
 \eqref{eq:1.A1}, Theorem \ref{thm:4.2} and the Poincar\'e inequality, we
 obtain
 \begin{equation}
  -
   \int_0^1
   \frac{\mu\sigma(\alpha)}{1+|u_x|^2}u^2_{xxx}\,dx
   \leq
   -
   \Cr{const:5.4}
   \int_0^1
   |u_{xx}(t,x)|^2\,dx,
 \end{equation}
 where $\Cl{const:5.4}>0$ is a positive constant depending only on
 $\mu$, $\Cr{const:1.1}$, $\sigma(\alpha_0)$,
 and $\sup_{x\in\mathbb{T}}(1+u_{0x}^2(x))$. By \eqref{eq:5.11}, we obtain
 \begin{equation*}
   \frac{d}{dt}
   \int_0^1
   |u_{xx}(x,t)|^2\,dx
   \leq
   -
   \Cr{const:5.4}
   \int_0^1
   |u_{xx}(t,x)|^2\,dx
   +
   \Cr{const:5.3}
   e^{-\Cr{const:5.1}t}
   \int_0^1
  |u_{0x}(x)|^2\,dx.
 \end{equation*}
 By the Gronwall inequality, we obtain \eqref{eq:5.12}.
\end{proof}

Next, we show exponential decay for
$\|u_{xxx}(\cdot,t)\|_{L^2(0,1)}$. We need the Schauder estimates for
the higher derivatives.

\begin{proposition}
 Let $(u,\alpha)$ be a time global solution of \eqref{eq:2.13} with the
 same assumptions as for Theorem \ref{thm:4.5}. Then, there is a constant
 $\Cl{const:5.5}>0$ depending only on $\gamma$, $\mu$, $\varepsilon$,
 $L$, $M$, $\Cr{const:1.1}$, and $\sigma(\alpha(0))>0$ such that
 \begin{equation}
  \label{eq:5.13}
  \|u_{x}\|_{C^{2,\beta}(Q_T^{\varepsilon})}
   \leq \Cr{const:5.5}.
 \end{equation}
\end{proposition}

\begin{proof}
 We let $w=u_x$. Then $w$ satisfies
 \begin{equation}
  w_t
   =
   \mu\sigma(\alpha(t))
   \left(
    \frac{1}{1+u_x^2}w_{xx}
    -
    \frac{2u_xu_{xx}}{(1+u_x^2)^2}w_x
   \right).
 \end{equation}
 With $u$ satisfying \eqref{eq:4.24}, we can apply the Schauder
 estimates \cite[Theorem 4.9]{MR1465184}. There is then a constant
 $\Cr{const:5.5}>0$ such that
 \begin{equation}
  \|u_x\|_{C^{2,\beta}(Q_T^\varepsilon)}
   =
   \|w\|_{C^{2,\beta}(Q_T^\varepsilon)}
   \leq
   \Cr{const:5.5}.
 \end{equation}
\end{proof}

Using the Schauder estimates, \eqref{eq:5.13}, and similar arguments in
Proposition \ref{prop:5.6}, we obtain

\begin{proposition}
 \label{prop:5.8}
 Let $(u,\alpha)$ be a classical solution of \eqref{eq:2.13}. Then
 there exists $\Cl{const:5.6}>0$ such that
 \begin{equation}
  \label{eq:5.14}
   \int_0^1
   |u_{xxx}(x,t)|^2\,dx
   \leq
   e^{-\Cr{const:5.6}t}
   \int_0^1
  |u_{0xxx}(x)|^2\,dx
 \end{equation}
 for $t>0$.
\end{proposition}
Finally, we prove the asymptotic behavior of the global solution.
\begin{proof}%
 [Proof of Theorem \ref{thm:5.1}]
Using Proposition \ref{prop:5.8} and the Sobolev inequality, we obtain
\begin{equation}
 \label{eq:5.15}
 |\kappa(x,t)|
  \leq
  |u_{xx}(x,t)|
  \leq
  \int_0^1
  |u_{xxx}(x,t)|\,dx
  \leq
  \left(
   \int_0^1
   |u_{xxx}(x,t)|^2\,dx
  \right)^\frac12
  \leq
  \Cr{const:5.7}
  e^{-\frac{\Cr{const:5.6}}2t}
\end{equation}
for some $\Cl{const:5.7}>0$. Thus $u_{xx}$ and  curvature $\kappa$ go
to $0$ exponentially and uniformly on $[0,1)$. In addition, we can show
that $u_x$ converges to $0$ exponentially and uniformly on $[0,1)$,
similarly.  Therefore we only need to prove that there exists a constant
$u_\infty$ such that $u $ goes to $u_\infty$ exponentially and uniformly
on $[0,1)$. For any $0 \leq t_1 < t_2$ and $x \in [0,1)$, we have
 \begin{equation*}
 \begin{split}
 &|u(x,t_2) -u(x,t_1)| \leq \int _{t_1} ^{t_2} |u_t (x,s)| \, ds
 \leq \int _{t_1} ^{t_2} \mu \sigma (\alpha (s)) \frac{|u_{xx} (x,s)|}{1+|u_x (x,s)| ^2} \, ds\\
 \leq & \, \mu \max_{|\alpha| \leq |\alpha_0|} \sigma (\alpha) \int _{t_1} ^{t_2} |u_{xx} (x,s)| \, ds
 \leq \mu \max_{|\alpha| \leq |\alpha_0|} \sigma (\alpha) \int _{t_1} ^{t_2} \Cr{const:5.7} e^{-\frac{\Cr{const:5.6}}{2} s} \, ds \\
 \leq &\mu \max_{|\alpha| \leq |\alpha_0|} \sigma (\alpha) \frac{2\Cr{const:5.7}}{C_{12}} e^{-\frac{\Cr{const:5.6}}{2} t_1},
 \end{split}
 \end{equation*}
 where \eqref{eq:4.31} and \eqref{eq:5.15} are used. Hence, there exists
 $u_\infty=u_\infty (x)$ such that $u$ goes to $u_\infty$ exponentially
 for any $x \in [0,1)$. In addition, with $u_x$ converging to $0$
 uniformly, $u_\infty$ should be a constant.  Consequently, $u$
 converges to constant $u_\infty$ exponentially and uniformly on
 $[0,1)$.
\end{proof}

\section*{Acknowledgment}

The authors thank the referees for their careful reading of the paper
and helpful comments. This work was supported by JSPS KAKENHI Grant
Number JP20K14343, JP18K13446, JP18H03670, JP16K17622 and Leading
Initiative for Excellent Young Researchers (LEADER) operated by Funds
for the Development of Human Resources in Science and Technology.



\providecommand{\bysame}{\leavevmode\hbox to3em{\hrulefill}\thinspace}

\end{document}